\documentclass[12pt]{amsart}
\usepackage{amsbsy,amssymb,amscd,amsmath,amsthm, hyperref}
\usepackage{a4wide}
\usepackage[usenames,dvipsnames,svgnames,table]{xcolor}

\newcommand{\out}[1]{\ }

\newcommand{\RR}{{\mathbb R}}
\newcommand{\CC}{{\mathbb C}}

\newcommand{\abs}[1]{|#1| }
\renewcommand{\ge}{\geqslant}
\renewcommand{\le}{\leqslant}
\renewcommand{\phi}{\varphi}
\renewcommand{\epsilon}{\varepsilon}

\newcommand{\eps}{\varepsilon}

\DeclareMathOperator{\diam}{diam}

\DeclareMathOperator{\PSH}{PSH}

\DeclareMathOperator{\PS}{PS}

\makeatletter
\newtheorem*{rep@theorem}{\rep@title}
\newcommand{\newreptheorem}[2]{%
\newenvironment{rep#1}[1]{%
 \def\rep@title{#2 \ref{##1}}%
 \begin{rep@theorem}}%
 {\end{rep@theorem}}}
\makeatother

\newtheorem{theorem}{Theorem}[section]
\newreptheorem{theorem}{Theorem}

\newtheorem{lemma}[theorem]{Lemma}

\theoremstyle{definition}
\newtheorem{definition}[theorem]{Definition}
\newtheorem{example}[theorem]{Example}
\newtheorem{question}[theorem]{Question}

\theoremstyle{remark}
\newtheorem{remark}[theorem]{Remark}

\numberwithin{equation}{section}

\newtheoremstyle{case}
{3pt}
  {3pt}
  {}
  {}
  {\bfseries}
  {:}
  {.5em}
  {}
\theoremstyle{case}
\newtheorem{case}{Case}
\newtheorem{subcase}{Case}
\numberwithin{subcase}{case}

\begin{document}

\title{A note on approximation of plurisubharmonic functions}

\dedicatory{Dedicated to John Erik Forn\ae ss on occasion of his 70-th birthday.}
\author{H\aa kan Persson}
\address{Dept. of Maths. Uppsala University\\SE-751 06 Uppsala\\Sweden}
\email{hakan.persson@math.uu.se}
\author{Jan Wiegerinck}
\address{KdV Institute for Mathematics
\\University of Amsterdam
\\Science Park 105-107
\\P.O. box 94248, 1090 GE Amsterdam
\\The Netherlands}
\email{j.j.o.o.wiegerinck@uva.nl}

\begin{abstract} We extend a recent result of Avelin, Hed, and Persson about approximation of functions $u$ that are plurisubharmonic on a domain $\Omega$ and continuous on $\bar\Omega$, with functions that are plurisubharmonic on (shrinking) neighborhoods of $\bar\Omega$. We show that such approximation is possible if the boundary of $\Omega$ is $C^0$ outside a countable exceptional set $E\subset\partial \Omega$. In particular, approximation is possible on the Hartogs triangle. For H\"older continuous $u$, approximation is possible under less restrictive conditions on $E$. We next give examples of domains where this kind of approximation is not possible, even when approximation in the H\"older continuous case is possible.
\end{abstract}

\maketitle

\section{Introduction}\label{sec0}
In \cite{Sib} Sibony showed that any $C^\infty$-smooth pseudoconvex domain $\Omega$ in $\CC^n$ has  the \emph{P(luri)S(ub)H(armonic) Mergelyan property}, cf.  \cite{Hed}.  That is, every function $u$  that is plurisubharmonic on $\Omega$  and continuous on $\overline \Omega$ can be approximated uniformly on $\overline \Omega$ by continuous plurisubharmonic functions $v_j$ that are defined on (shrinking) domains $\Omega_j$ that contain $\overline \Omega$. Forn\ae ss and the second author extended this result in \cite {FoWi} to arbitrary domains with $C^1$-boundary. Recently, see \cite{Hed, AHP}, it was shown that the result remains valid if one only assumes  $C^0$-boundary.

Recall that a domain $G$ in $\RR^n$ has $C^0$-boundary if the boundary is locally the graph of a continuous function over an $(n-1)$-dimensional hyperplane. This is equivalent to saying that $G$ has the \emph{Segment Property} for every $z$ in $ \partial G$, cf.~\cite{Fr, Hed, AHP}, i.e.,  there exists a neighborhood $U$ of $z$ and a vector $w\in\RR^n$ such that 
\[U\cap\overline G+tw\subset G,\quad \text{ for all }0<t<1.\]

For an arbitrary domain in $G$ we will call $z\in\partial G$ a $C^0$-boundary point if $G$ has the Segment Property at $z$. The $C^0$-boundary points form a relative open subset of $\partial G$.

\smallskip
Let $\Delta$ be the unit disc in $\CC$. In \cite{Hed} it is observed that the domain $\Delta\setminus [-1/2,1/2]\}$ does not have the PSH-Mergelyan property.   The natural setting to study this kind of approximation problems, however, is  the case where $\Omega$ is a fat domain, i.e. $\Omega=(\overline\Omega)^o$. For domains in $\CC$ (and in fact $\RR^n$) results from classical potential theory give precise criteria for "Mergelyan type" approximation with \emph{harmonic} and \emph{subharmonic} functions. For harmonic functions this goes already back to Keldysh, \cite{Keld}, and Deny, \cite{Den}, see also  \cite{ArGa, Gar, Hedb}. We recall the result in Section \ref{Sec3}, where we give several examples of domains where the PSH-Mergelyan property does not hold. 

There is a big gap between the counterexamples in $\CC$ and domains with $C^0$-boundary. It is a natural question whether the Hartogs triangle, $R=\{(z,w): 0<|z|<|w|<1\}$ has the PSH-Mergelyan property, as it is known that $R$ does not have the Mergelyan property for holomorphic functions. As the main result of this note, however, we will  prove the following theorem.
\begin{reptheorem}{thm2}
Domains in $\CC^n$ of which the boundary is  $C^0$ with the possible exception of a countable set of boundary points,  have the PSH-Mergelyan property. 
\end{reptheorem}
In particular this shows that the Hartogs triangle has the PSH-Mergelyan property.

In Section \ref{sec:holder}, we study uncountable exceptional sets and show that when extra assumptions are made on the modulus of continuity of the approximand, the result of Theorem \ref{thm2} can be extended to certain uncountable sets. In Section \ref{Sec3}, we give several examples showing the difficulty of characterising the PSH-Mergelyan property. We also make some comparisons to the classical problem of (sub-) harmonic approximation and to approximation involving H\"older continuous functions.

\section{Approximation of plurisubharmonic functions} Let $\PSH(\Omega)$ denote the plurisubharmonic functions on $\Omega$ and $\PSH(\overline \Omega)$ the functions that are uniform limits of plurisubharmonic functions defined on (shrinking) neighborhoods of $\overline\Omega$. Let $B(a,r)\subset\CC^n$ denote the ball with center $a$ and radius $r$.

For the readers convenience, we recall the familiar definition of modulus of continuity.
\begin{definition}
	Suppose that $u$ is a function defined on the set $X$. A function $\omega:[0,\infty] \to [0,\infty]$ is called a \emph{modulus of continuity} for $u$ if
    \[
    	|u(z)-u(w)| \leq \omega\bigl(|z-w|\bigr), \quad \forall z,w \in X.
    \]
\end{definition}
\begin{lemma}\label{lem2}
Let $D$ be a bounded domain in $\mathbb C^n$ and let  $u$ be in $\PSH(D)\cap C(\bar D)$ with modulus of continuity $\omega$. Let $a \in D$. For every small $\epsilon>0$ there exists a function $v\in \PSH(D) \cup B(a, \epsilon)$ that is continuous on $D \cup B(a, \epsilon)$ and satisfies
\[
	\sup_{\bar D}|u -v|\leq 4\omega\Bigl(\diam(D)^{2/3}\epsilon^{1/3}\Bigr).
\]

\end{lemma}
\begin{proof}
	Without loss of generality,  assume that $a$ is the origin and that $f(0) = 0$.
	
	Let $\eta=\diam(D)^{2/3}\epsilon^{1/3}$, and 
	\[
		\alpha=6\omega(\eta)\Bigr(\log\bigl(\diam(D)/\epsilon\bigr)\Bigr)^{-1}.
	\]
	 Define
	\[
		v(z)=\begin{cases}
			\max\{u(z)+\alpha \log\bigl| z/\eta\bigr|,-2\omega(\eta)\},&\mbox{if $|z|\leq \eta$};\\
			u(z)+\alpha\log\bigl| z/\eta\bigr|,&\mbox{if $|z|> \eta$}.
			\end{cases}
	\]
	Since $u(z)+\alpha\log\bigl| z/\eta\bigr|\geq - \omega(\eta)$ for $|z|=\eta$ it is obvious that $v\in \PSH(D)\cap C(\bar D)$. Furthermore, when $|z|\leq \epsilon$,
	\begin{align*}
		u(z)+\alpha \log\bigl| z/\eta\bigr|	&\leq \omega(\epsilon)+\alpha\log\bigl(\epsilon/\eta\bigr)\\
			&=\omega(\epsilon)+\frac{6\omega(\eta)}{\log\bigl(\diam(D)/\epsilon\bigr)}\log\bigl(\epsilon/\eta\bigr)\\
			&=\omega(\epsilon)+\frac{6\omega(\eta)}{\log\bigl(\diam(D)/\epsilon\bigr)}\log\left(\bigl(\epsilon/\diam(D)\bigr)^{2/3}\right)\\
			&\leq \omega(\epsilon)-4\omega(\eta)\\
			&\leq -3\omega(\eta),
	\end{align*}
	and hence $v\equiv -2\omega(\eta)$ on $D \cap B(a,\epsilon)$. Thus setting $v\equiv -2\omega(\eta)$ extends $v$ as a plurisubharmonic function to $B(a,\epsilon)$. Furthermore, if $\epsilon$ is small and $v(z)=u(z)+\alpha\log\bigl| z/\eta\bigr|$, it follows that
	\begin{align*}
		|u(z)-v(z)|	&\leq \alpha\Bigl|\log\bigl| z/\eta\bigr| \Bigr|\\
				& = \alpha\Bigl|\log|z|-\frac{2}{3}\log(\diam(D))-\frac{\log(\epsilon)}{3}\Bigr|\\
				&\leq \frac{2\alpha}{3} \Bigl|\log(\diam(D))-\log(\epsilon)\Bigr|\\
				&=4\omega(\eta)=4\omega\Bigl(\diam(D)^{2/3}\epsilon^{1/3}\Bigr)
	\end{align*}
	and if $v(z)=-2\omega(\eta)$, it follows that
	\begin{align*}
		|u(z)-v(z)| &\leq \omega(\eta)+2\omega(\eta)\\
				&= 3\omega\Bigl(\diam(D)^{2/3}\epsilon^{1/3}\Bigr).
	\end{align*}
\end{proof}

\begin{theorem}\label{thm2}Let $\Omega$ be a bounded domain in $\CC^n$. Suppose that $\Omega$ has $C^0$ boundary except at  countable set of boundary points $K$. Then $\Omega$  has the PSH-Mergelyan property.
\end{theorem}
\begin{proof}Because the set of $C^0$-boundary points is open, $K$ is compact. Let $u \in\PSH(\Omega)\cap C(\overline\Omega)$ and let $K=\{a_1,a_2,\ldots\}$. Let $\eps>0$. We apply Lemma \ref{lem2}  on a subsequence $(a_{n_j})$ of $K$ as follows. Applying the lemma to $a_{n_1}=a_1$ we find a ball $B_1=B(a_{n_1}, r_1)$ and a continuous  function $v_1$ defined on the closure of $\Omega_1=\Omega\cup B_1$ that is plurisubharmonic on $\Omega_1$ such that $|v_1-u|\le\epsilon/2$ on $\overline\Omega$ and such that $K\cap\partial B_1=\emptyset$.

Let $a_{n_2}$ be the first element of $K$ that is not in $B_1$. Then $a_{n_2}\notin \overline{B_1}$ and applying the lemma gives us a ball $B_2$ about $a_{n_2}$ and a continuous function $v_2$ defined on the closure of $\Omega_2=\Omega_1\cup B_2$ that is plurisubharmonic on ${\Omega_2}$ and satisfies $|v_1-v_2|\le e^{-2}\eps$ on $\overline{\Omega_1}$ and such that $\overline {B_2}\cap \overline {B_1}=\emptyset$ and $K\cap \partial B_2=\emptyset$. 

We continue in this fashion and obtain a sequence of balls $B_j$, and continuous functions $v_j$, which are defined on the closure of $\Omega_j=\Omega\cup\bigcup_1^jB_i$ and plurisubharmonic on  $\Omega_j$,  such that $|v_{j-1}-v_j|<2^{-j}\eps$ on $\Omega_j$,  $K\subset \cup_j B_j$ and $\overline{B_i}\cap\overline{B_j}=\emptyset$ for $i\ne j$. Because $K$ is compact, this sequence is finite,  $K\subset \cup_{j=1}^N B_j$ for some $N>0$. The function $v_N$ is defined and plurisubharmonic on $\Omega_N$ and $|v_N-u|<\eps$ on $\overline \Omega$. We will use Gauthier's Localization Theorem, see \cite{Ga}, also  cf.  \cite{Hed}, Chapter 5, which states that $f\in C(\overline{\Omega})\cap\PSH(\overline{\Omega})$ if and only if for every $z_0\in \overline\Omega$  there is an open set $U$ containing $z_0$ such that $f\in C(\overline{\Omega\cap U})\cap\PSH(\overline{\Omega\cap U})$. We apply this to $v_N$.
For $z_0\in \overline{\Omega}\cap B(a_j,r_j)$ or  $z_0\in\Omega$ this is immediate. If $z_0\in\partial\Omega\setminus\cup_j B(a_j,r_j)$ then $\Omega$ has by assumption the segment property at $z_0$.  Then  there exists a neighborhood $U$ of $z_0$ and a vector $w\in\CC^n$ such that 
\[U\cap\overline\Omega+tw\subset\Omega,\quad \text{ for all }0<t<1.\]
Fix a ball $B(z,r)$ that is compactly contained in $U$. Then  the functions $v_t$ defined by $v_t(z)=v_N(z+tw)$, which are defined on neighborhoods of $\overline\Omega\cap B(z_0,r)$ approximate $v_N$ uniformly on $\overline\Omega\cap\overline{B(z,r)}$ when $t\downarrow 0$. The conclusion is that $v_N$ and hence also $u$ belongs to $C(\overline{\Omega})\cap\PSH(\overline{\Omega})$.
\end{proof}

\begin{remark}Note that this proof also gives a proof of the PSH-Mergelyan property for $C^0$-domains. It is shorter, because of Gauthier's theorem, but essentially not different from the proof in \cite{AHP}. 
\end{remark}
\section{Approximation of H\"{o}lder continuous PSH functions}\label{sec:holder}
In many practical cases when working with plurisubharmonic functions, one has some extra quantitative information about the modulus of continuity (see for example \cite{DF} and \cite{Sic}). In such cases we can extend the results of Theorem \ref{thm2} to certain uncountable exceptional sets. In order to describe those sets, we need the following definition.
\begin{definition}
	Let $C>1$ and $\phi:\mathbb R^+ \to \mathbb R^+$. The set $K$ satisfies $\PS(C, \phi)$ ($K$ is said to be \emph{C-proportionally well separated of negligible $\phi$-measure}), if for each $\epsilon>0$ there exists a set of balls $\mathcal F =\{B(z_j,r_j)\}$ satisfying the following properties:
	\begin{enumerate}
		\item $\mathcal F$ covers $K$;
		\item
			$\mathcal F$ is proportionally $C$-separated in the sense that 
			\[
				B(z_j,r_j)\cap B\bigl(z_k,Cr_k\bigr)=\emptyset,\quad \mbox{ for all $j\ne k$};
			\]
		\item
			\[
				\sum_{j=1}^\infty \phi(r_j)<\epsilon.
			\]
	\end{enumerate}
\end{definition}
\begin{remark}
Note that $K$ satisfies $\PS(C,\phi)$ implies that the generalized Hausdorff measure $\Lambda_\phi(K)=0$.
\end{remark}

\begin{example}\label{ex:cantor2}
	Let ${\mathbf s}=\{s_1,s_2,\cdots\}$ be a sequence of positive numbers $0 <s_j<1$. Let $I_0$ be a closed interval in $\RR$ and $C(s_1)$ be $I_0\setminus J_1]$ where $J_1$ is an open interval about the center of $I_0$ of length $(1-s_1)|I_0| $. Proceeding by induction, $C(s_1,\ldots, s_n)$ is obtained by removing from each closed interval $I$ in $C(s_1,\ldots,s_{n-1})$ an open interval of length $(1-s_n)|I|$ about the center of $I$. The generalized Cantor set $C({\mathbf s})$ is now defined as follows.
     \[C({\mathbf s})=\bigcap_{n\ge 1}C(s_1,\ldots. s_n).
     \]
     
     It is well known that $C({\mathbf s})$ is homeomorphic to the standard Cantor set, and that its capacity is positive if and only if
     \begin{equation}\label{poscap}\sum_{n=1}^\infty\frac{\log s_n }{2^n}
     \end{equation}
     is finite, cf.~\cite{Nev,Ran}.

     Now take ${\mathbf s}=\{2^{-n}\}_{n=1}^\infty$. Then by \eqref{poscap} $C(\mathbf{s} )$ has positive capacity.
 We claim that for every $\delta>0$ the set $C(\mathbf{s} )$ satisfies $\PS(\phi,C)$ for every $C$ and $\phi(s)=s^\delta$.
 
 Indeed, note that $C(2^{-1},\ldots 2^{-N})$ consists of $2^N$ intervals  of length $L_N=\prod_1^N (s_j/2)=2^{-N(N+3)/2}$. Each of these intervals can be covered by an interval with the same midpoint and radius $r_N=L_N$. Thus $C(\mathbf{s} )\subset C(2^{-1},\ldots 2^{-N})$ is covered by $2^N$ intervals $I_j$ of radius $r_N$.  The distance between
 $I_j$ and $CI_k$ is smallest when $I_j$ and $I_k$ belong to the same interval in $C(2^{-1},\ldots 2^{-(N-1)})$. This distance then equals
 \[L_{N-1} - (2+C)r_N=L_{N-1}-(2+C)2^{-N-1}L_{N-1},\]
 which is positive if $C<2^{N+1}-2$.
 
For fixed $\delta>0$
	\begin{align*}
		\sum_{j=1}^\infty r_j^\delta	&=2^N \left(2^{(-N(N+3))/2}\right)^\delta\\
							&\leq 2^{N-\delta(N^2+3N)/2},
	\end{align*}
    which tends to 0 if $N\to\infty$. 
\end{example}

The following example shows that it might be hard to give sufficient criteria for sets to be proportionally separated.

\begin{example}
	The set $X \subset \mathbb R$ defined by $X=\{1/j:j\in \mathbb N\}$ satisfies Property $\PS(C,\phi)$ with $C=2$ and $\phi(t)=t^\alpha$ for all $\alpha >0$. On the other hand the set $Y=X \cup \{0\}$ does not satisfy Property $\PS(C,t^\alpha)$ for any $C>1$ and $\alpha$. 
\end{example}
We are now ready to state and prove the main result of this section.
\begin{theorem}\label{thm:holder}
	Suppose that $D \subset \mathbb C^n$ is a bounded domain and that $u \in \PSH(D)\cap C(\bar D)$, with a concave modulus of continuity $\omega$. Suppose also that there is a set $E\setminus \partial D$ such that each $z \in \partial \Omega \setminus E$ has a neighborhood $U \subset \mathbb C^n$ such that $u \in \PSH(\overline{U\cap D})$.	 

If $E$ satisfies $\PS(C,\phi)$ with $\phi(t)=\omega(t)\log(1/t)$, then $u \in \PSH(\overline D)$.
\end{theorem}
\begin{remark}
Since $\overline \Omega$ is compact and $u$ is continuous, $u$ admits \emph{at least} one concave modulus of continuity.
\end{remark}

\begin{remark}
	One might wonder if the assumption on the modulus of continuity is necessary, or just a consequence of our method of proof. Although we don't know that this condition in sharp, Example \ref{example1} of the subsequent section shows that some condition on the modulus of continuity is necessary for the theorem to hold.
\end{remark}

\begin{proof}
Without loss of generality,  assume that $\diam(D)=1$. For a fixed $\epsilon >0$, let $\{B(z_j,r_j\}_{j=1}^N$ be a covering of $K$ satisfying
\[
	\sum_{j=1}^N\omega(r_j)\log(1/r_j)\leq \epsilon,
\]
such that $\{B(z_j,r_j\}_{j=1}^N$ is pairwise disjoint and $B(z_j,r_j)\cap B(z_k,Cr_k)=\emptyset$ for all $j\ne k$ and some $1<C<2$.

Now let
	\[
		v_j(z)=\begin{cases}
			\max\{u(z)+\frac{2}{\log(C)}\omega(Cr_j) \log\bigl| \frac{z-z_j}{Cr_j}\bigr|,u(z_j)-\omega(Cr_j)\},&\mbox{if $|z-z_j|\leq Cr_j$}\\
			u(z)+\frac{2}{\log(C)}\omega(Cr_j) \log\bigl| \frac{z-z_j}{Cr_j}\bigr|,&\mbox{if $|z-z_j|> Cr_j$}
			\end{cases}
	\]
If $|z-z_j|=Cr_j$, it follows that
\[
	u(z)+2\omega(Cr_j)\log\bigl| \frac{z-z_j}{Cr_j}\bigr|\geq u(z_j)- \omega(Cr_j)
\]
and therefore it follows by the usual gluing argument that $v\in \PSH(D)\cap C(\bar D)$. Furthermore, when $\abs{z-z_j}\leq r_j$,
	\begin{align*}
		u(z)+\frac{2}{\log(C)}\omega(Cr_j) \log\bigl|\frac{z-z_j}{Cr_j}\bigr|	&\leq u(z_j)+\omega(r_j)+\frac{2}{\log(C)}\omega(Cr_j)\log\bigl(r_j/(Cr_j)\bigr)\\
			&=u(z_j)+\omega(r_j)-\frac{2}{\log(C)}\omega(Cr_j)\log(C)\\
			&\leq u(z_j)-\omega(Cr_j),
	\end{align*}
	and hence $v_j$ is constant on $D \cap B(z_j,r_j)$. Thus $v_j$ can be trivially plurisubharmonically extended to $D\cup B(z_j,r_j)$.
	
	Now let
	\[
		\tilde v_j(z)=v_j(z)+\frac{2}{\log(C)}\sum_{k\neq j}\omega(Cr_k) \log\bigl| \frac{z-z_k}{Cr_k}\bigr|,
	\]
and define
\[
	v(z)=\max\{\tilde v_j(z): 1\leq j \leq N\}.
\]
We now want to show that $v(z)=\tilde v_j(z)$ when $|z-z_j|<Cr_j$, which will imply that $v$ admits a plurisubharmonic extension to $D\cup B(z_j,r_j)$. For this, suppose that $|z-z_j|<Cr_j$ and let $k\neq j$ be arbitrary. Then it follows from the fact that  $B(z_j,Cr_j)\cap B(z_k,Cr_k)=\emptyset$, that 
\begin{align*}
	\tilde v_k(z)	&=u(z)+\frac{2}{\log(C)}\sum_{k=1}^{N}\omega(Cr_k) \log\bigl| \frac{z-z_k}{Cr_k}\bigr|\\
				&\leq v_j(z)+\frac{2}{\log(C)}\sum_{k\neq j}\omega(Cr_k) \log\bigl| \frac{z-z_k}{Cr_k}\bigr| \\
				&=\tilde v_j(z).
\end{align*}
To see that $v$ approximates $u$, we consider two different cases.

\begin{case}
\[
	z \in \Omega \cap \bigcup_{j=1}^N B(z_j, Cr_j).
\]
Suppose that $z \in B(z_j,Cr_j)$ for some $j = 1 \ldots N$. Then it was earlier demonstrated that 
\[
	v(z)=v_j(z)+\frac{2}{\log(C)}\sum_{k\neq j}\omega(Cr_k) \log\bigl| \frac{z-z_k}{Cr_k}\bigr|.
\]
There are now two possibilities.
\begin{subcase}
If 
\[
	v_j(z)=u(z)+\frac{2}{\log(C)}\omega(Cr_j) \log\bigl| \frac{z-z_j}{Cr_j}\bigr|,
\]
then $r_j\le|z-z_j|\le C r_j$ and hence $\log\bigl| \frac{z-z_\ell}{Cr_\ell}\bigr|$ is negative for $\ell = j$ and positive for all other $\ell$. Keeping this in mind and that $\diam(D)< 1$, we see that
\begin{align*}
\tilde v_j(z)-u(z)&= \frac{2}{\log(C)}\sum_{k=1}^N\omega(Cr_k) \log\bigl| \frac{z-z_k}{Cr_k}\bigr|  \\
	&\leq \frac{2}{\log(C)}\sum_{k\neq j}\omega(Cr_k) \log\left(\frac{1}{Cr_k}\right)\\
	&\leq \frac{2}{\log(C)}\sum_{k\neq j}\omega(Cr_k)\log\left(\frac{1}{r_k}\right)\\
	&\leq \frac{2C}{\log(C)}\sum_{k\neq j}\omega(r_k)\log\left(\frac{1}{r_k}\right)\\
	&\leq \frac{2 C}{\log(C)}\epsilon,
\end{align*}
where we in the second to last inequality have used the fact that $\omega$ is concave.

On the other hand, it also holds that
\begin{align*}
\tilde v_j(z)-u(z)&= \frac{2}{\log(C)}\sum_{k=1}^N\omega(Cr_k) \log\left| \frac{z-z_k}{Cr_k}\right|  \\
    &\geq \frac{2}{\log(C)} \omega(Cr_j) \log\left|\frac{z-z_j}{Cr_j}\right|\\
    & \geq \frac{2}{\log(C)} \omega(Cr_j) \log\left(\frac 1{C}\right)\\
			&=-2\omega(Cr_j)\geq -2C\omega(r_j)\geq -4\epsilon.
\end{align*}
\end{subcase}
\begin{subcase}
If $v_j(z)=u(z_j)-\omega(Cr_j)$, it similarly follows that
\begin{align*}
	u(z)-\tilde v_j(z)&\leq u(z)- u(z_j)+\omega(Cr_j) -\frac{2}{\log(C)}\sum_{k\neq j}\omega(Cr_k) \log\bigl| \frac{z-z_k}{Cr_k}\bigr|\\
 	&\leq 2\omega(Cr_j)\leq 2C\epsilon,
\end{align*}
and
\begin{align*}
	\tilde v_j(z)-u(z)&\leq  u(z_j)-u(z)-\omega(Cr_j) +\frac{2}{\log(C)}\sum_{k\neq j}\omega(Cr_k) \log\bigl| \frac{z-z_k}{Cr_k}\bigr|\\
				&\leq  \omega(Cr_j)-\omega(Cr_j)+\frac{2}{\log(C)}\sum_{k\neq j}\omega(Cr_k) \log(1/Cr_k)\\
				&\leq  \frac{2C}{\log(C)}\sum_{k\neq j}\omega(r_k) \log(1/r_k)\\
				&\leq \frac{2C}{\log(C)}\epsilon.
\end{align*}
\end{subcase}
\end{case}
\begin{case}
	\[
	z \in \Omega \setminus \bigcup B(z_j, Cr_j).
	\]
		In this case note that now for all $j$ it holds that $v_j(z)=u(z)+\frac{2}{\log(C)}\omega(Cr_j) \log\bigl| \frac{z-z_j}{Cr_j}\bigr|$. It follows similarly as above that
\begin{align*}
	|u(z)-\tilde v_j(z)|&=\frac{2}{\log(C)}\sum_{k=1}^N\omega(Cr_k) \log\bigl| \frac{z-z_k}{Cr_k}\bigr|  \\
	&\leq \frac{2C}{\log(C)}\sum_{k=1}^N\omega(r_k)\log\left(\frac{1}{r_k}\right)\\
	&\leq \frac{2 C}{\log(C)}\epsilon.
\end{align*}
\end{case}
We have thus shown that on $\overline \Omega$, $u$ is a uniform limit of functions like $v$.  To show that $u\in \PSH(\bar D)$, it therefore suffices to show that $v\in \PSH(\bar D)$. It follows from Gauthier's localization lemma, that it is enough to show that every $z \in \partial D$ has a neighborhood $U_z$ such that $v \in \PSH(\overline{U_z\cap D})$. For $z \in \partial D\setminus K$, this follows from the assumptions of the theorem and for $z \in K$, it follows by the construction that there is a neighborhood $U_z$ of $z$ where $v$ is constant and hence trivially in $\PSH(\overline{U_z\cap D})$.
\end{proof}
\section{Domains without the Mergelyan property}  \label{Sec3}
The following result, which was  mentioned in the introduction, goes back to Keldysh and Deny in case of harmonic functions and is in full generality a consequence of results by Bliedtner and Hansen, \cite{BH}, and cf. \cite{Gar}. It reads as follows.
\begin{theorem}\label{KD}
Let $K$ be a compact subset of $\RR^n$. The following are equivalent
\begin{enumerate}
\item Every function  $u$ that is  continuous on $K$ and (sub)harmonic on $K^o$ can be uniformly approximated on $K$ by functions that are (sub)harmonic on (shrinking) neighborhoods of $K$.
\item The sets $\RR^n\setminus K$ and $\RR^n\setminus K^o$ are thin at the same points.
\end{enumerate}
\end{theorem}
 With this result, examples like the following have been constructed. 
\begin{example} \label{example0} Let $\{a_j\}$ be a sequence in $\Delta\setminus[-1/2,1/2]$, such that $\overline{\{a_j\}}=\{a_j\}\cup [-1/2,1/2]$. Define $c_j>0$ sequentially such that $h(z)=\sum_j c_j\log|z-a_j|>-1$ on $[-1/2,1/2]$, and next $r_j$ in such a way that  $h<-2$ if $|z-a_j|<r_j$ and, moreover the discs $\{|z-a_j|\le r_j\}$ are disjoint. Then 
 \[D=\Delta\setminus \overline{\left(\cup_j\{|z-a_j|<r_j\}\right)}=\Delta\setminus\left([-1/2,1/2]\cup\bigcup_j\{|z-a_j|\le r_j\}\right)\] 
is a fat domain in $\CC$, $K=\overline D$, and $K^o=D$ do not satisfy condition (2). In fact $\CC\setminus K= \bigcup_j\{|z-a_j|<r_j\}$ is thin at all points in $[-1/2,1/2]$, but $\CC\setminus K^o$ is not.
Hence $K$ does not have the approximation property (1).

For completeness, and because  Theorem \ref{KD} was misunderstood in \cite{HP} see also \cite[chapter 6]{Hed}, we show directly that $D$ is hyperconvex, but not $P$-hyperconvex, that is, $D$ admits a bounded subharmonic exhaustion function $g$, but no such function belongs to $\PSH(\overline D)$. This corrects Remark 4.10 in \cite{HP}. 

Observe that the domain $D$ is a regular for the Dirichlet problem, since it has no isolated boundary points. Denote by $F$ the solution of the Dirichlet problem on $D$ with boundary values $|z|^2$ on $\partial D$. Then the function $ z\mapsto |z|^2-F(z)$ is a bounded subharmonic exhaustion function. To see that no bounded subharmonic exhaustion function $g$ can be in $\PSH(K)$, notice that $[-1/2,1/2]$ is contained in the fine interior of $K$, because $[-1/2,1/2] \subset \{h>-3/2\}$, a finely open subset of $K$. Now if $g$ were in $\PSH(K)$, $g$ would be finely subharmonic on the fine interior of $K$ as uniform limit of subharmonic functions. But $g=0$ on $[-1/2,1/2]$ and $\le 0$ on a fine neighborhood of $K$, which would contradict the \emph{fine maximum principle}, cf.  ~\cite[Theorem 12.6]{Fu}.
 
Observe that this also shows that $D$ does not have the PSH-Mergelyan property either.

\end{example}

The example in \cite{Hed}, mentioned in the introduction, can be adapted to show that even if the exceptional set $K$ in the boundary is very small, the domain will in general not have the PSH-Mergelyan property.

\begin{example}\label{example1} (1) Let $C=C(\mathbf{s})$ ($\mathbf{s}=\{2^{-k}_{k\ge 1}\}$) in the interval $[-1/2,1/2]$ be the Cantor set of Example \ref{ex:cantor2}. We have seen that $C$ has  Hausdorff dimension 0 and positive capacity. Let $\Delta$ be the open disc with radius 2 in $\CC$, and let $D=\Delta \setminus C$. Then $D$ is regular for the Dirichlet problem, and hence there exists a continuous function $u$ on $\overline \Delta$ that is harmonic on $D$ satisfying $u|_{\{|z|=2\}}=0$ and $u|_{C}=1$. Observe that $u\notin \PSH(\overline D)\cap C(\overline D)$ because of the maximum principle. 

 As $D$ is a regular domain for the Dirichlet problem, it admits a bounded continuous subharmonic exhaustion function $\psi$ with $-1\le\psi\le 0$. 
We set
\[\Omega=\{(z,w)\in D\times \CC: |w|+\psi(z)<0\}.\]
By its definition $\Omega$ is hyperconvex and fat. It is $C^0$ at all boundary points of the form $(z, \psi(z)e^{i\theta})$, $z\in D$, $\theta\in[0,2\pi)$. The function $u$, now viewed as a function on $\overline \Omega$, is not in 
$\PSH(\overline \Omega)\cap C(\overline \Omega)$, again because of the maximum principle.

In fact, the functions  $u$ and $\psi$ in the example cannot not even be approximated by subharmonic functions that are H\"older continuous on $\Delta$. Indeed, by a theorem of Sadullaev and Yarmetov, \cite{SaYa}, such H\"older continuous functions would in fact be subharmonic on $\Delta$, again violating the maximum principle. 

Similarly, we find that with $\Omega$ and $u$ as above, $u$ is not a uniform limit on $\overline \Omega$ of plurisubharmonic functions on $\Omega$ that are H\"older continuous on $\overline\Omega$. On the other hand, by Theorem \ref{thm:holder}, every H\"older continuous function in $\PSH(\Omega)$ belongs to $\PSH(\overline \Omega)$.
\end{example}

 In view of Theorem \ref{KD} and Example \ref{example0} one might hope that e.g. conditions like: \emph{the sets $\CC^n\setminus K$ and $\CC^n\setminus K^o$ are (pluri)-thin at the same points}, or perhaps that $K\cap L$ has the PSH-Mergelyan property for all complex lines $L$,
 would be equivalent with the PSH-Mergelyan property. This is not the case.

\begin{example}Let $K=\{(z,w): z\in \overline \Delta, |w|\le d(z, [-1/2,1/2])\}$ then $\CC^2\setminus K$ and $\CC^2\setminus K^o$ are both thin and pluri-thin at points of $K^o$ and not thin or pluri-thin at other points of $\CC^2$. For points in $K\setminus K^o$ this follows from the Poincar\'e Zaremba Criterion: in such points $p$ there exists an open solid cone of revolution with vertex $p$ in $\CC^2\setminus K$, hence for a (pluri) subharmonic function $h$ defined in a neighborhood of $p$ we have $\limsup_{z\to p,z\notin K} h(z)=h(p)$.   However,  $K$ does not have the PSH-Mergelyan property.
\end{example}

\begin{example} Let $\Omega$ be the domain of Example \ref{example1}, we will write $w=u+iv$, and let $\tilde \Omega=(\Omega\cup \{v>0\})\cap \{|z|^2+|w|^2<3/2\}$. Then $\partial\tilde\Omega$ has $C^0$ boundary. This is clear at boundary points $(z,w)$ with $w\ne 0$ and at boundary points belonging to $\{|z|^2+|w|^2=3/2\}$. Define for $z\in\Delta$ and $|u|+\phi(z)\le 0$ the continuous function $\psi: (z, u)\mapsto-\sqrt{\phi^2(z)-u^2}$. Then the part of the boundary where $|z|^2<3/2$ and $v\le 0$ is the  graph $(z, u+i\psi(z,u))$ over a suitable domain in $D \times \RR$. 
Hence $\tilde\Omega$ has the PSH-Mergelyan property, but its intersection with the line $w=0$ has not.
\end{example}

\begin{question}
With an eye to Example \ref{example1} we ask the following. Suppose a domain $\Omega$ has $C^0$-boundary, except for a compact set $K$ in $\partial \Omega$. Suppose also that $K$ is a polar subset of a 1-dimensional analytic variety $J$. Does $\Omega$ have the PSH-Mergelyan property?
\end{question}

\begin{question} Suppose that $\Omega$ is a fat, bounded, contractible domain. Does $\Omega$ have the PSH-Mergelyan property?
\end{question}

\bibliographystyle{amsplain}

\end{document}